\documentclass{amsart}
\usepackage{color}
\usepackage{graphicx}
\usepackage{amsmath,amscd}
\usepackage{amssymb}
\input epsf
\usepackage{epsfig}

\usepackage{enumerate}
\newtheorem{theorem}{Theorem}[section]
\newtheorem{lemma}[theorem]{Lemma}
\newtheorem{corollary}[theorem]{Corollary}

\theoremstyle{definition}

\theoremstyle{remark}
\newtheorem{remark}[theorem]{Remark}

\numberwithin{equation}{section}
\newcommand{\Real}{{\mathbb R}}

\newcommand{\x}{\mathbf{x}}
\newcommand{\y}{\mathbf{y}}

\newcommand {\hide}[1]{}

\begin{document}
\title[Betti numbers for tropical prevarieties]
{Upper bounds on Betti numbers of tropical prevarieties}
{
\let\thefootnote\relax\footnote{2010 Mathematics Subject Classification 14T05}
}
%    Information for first author
\author{Dima Grigoriev}
%    Address of record for the research reported here
\address{CNRS, Math\'ematiques, Universit\'e de Lille, Villeneuve d'Ascq, 59655, France}
%    Current address
\email{dmitry.grigoryev@math.univ-lille1.fr}
%    \thanks will become a 1st page footnote.
%\thanks{}
%    Information for second author
\author{Nicolai Vorobjov}
\address{
Department of Computer Science, University of Bath, Bath
BA2 7AY, England, UK}
\email{nnv@cs.bath.ac.uk}
%\thanks{The first author was supported in part by NSF grant DMS-0801050.}
%\keywords{Homotopy type, Semi-algebraic sets}
% ----------------------------------------------------------------

\begin{abstract}
We prove upper bounds on the sum of Betti numbers of tropical prevarieties in dense
and sparse settings.
In the dense setting the bound is in terms of the volume of Minkowski sum
of Newton polytopes of defining tropical polynomials, or, alternatively,
via the maximal degree of these polynomials.
In sparse setting, the bound involves the number of the monomials.
\end{abstract}
\maketitle

% ----------------------------------------------------------------
\section*{Introduction}

In this paper we are concerned with upper bounds on Betti numbers of tropical prevarieties.
Basic definitions and statements regarding tropical algebra and geometry can be found in \cite{MS, RST}.

Each {\em tropical polynomial} $f$ in $n$ variables can be represented as $\min \{L_{1}, \ldots ,L_{m} \}$,
where $L_{1}, \ldots ,L_{m}$ are linear functions on $\Real^n$ (called {\em tropical monomials})
with non-negative integer coefficients at variables and real constant terms
(if negative integer coefficients are allowed, then $f$ is called {\em tropical Laurent polynomial}).
For a monomial $L_j$ its {\em degree} is the sum of integer coefficients at variables, the maximum among degrees
of the monomials is called the {\em (tropical) degree} of $f$.
With any tropical polynomial $f$ we associate a concave piece-wise linear function
$$L(\x):= \min_{1 \le j \le m}\{ L_j(\x) \}$$
defined on $\Real^n$.
A {\em tropical hypersurface} $V:=V(f) \subset \Real^n$ is the set of all points in $\Real^n$ at which
$L(\x)$ is not smooth.
Any point $\x \in V$ is called a {\em zero} of $f$.
A {\em tropical prevariety} is an intersection of a finite number of tropical hypersurfaces, in other words,
the set of all common tropical zeroes of a finite system of
multivariate tropical polynomials.

Let $V:=V(f_1, \ldots, f_k) \subset \Real^n$ be the tropical prevariety defined by a system $f_1, \ldots ,f_k$ of tropical polynomials
in $n$ variables of degrees not exceeding $d$.
In case $k=n$, the Tropical Bezout Theorem \cite{RST} states that the number of all {\em stable} tropical zeroes
(counted with multiplicities) of the system $f_1, \ldots ,f_k$ is the product of the degrees of its polynomials.
For arbitrary $k$, the upper bound
$$\binom{k+7n}{3n}d^{3n}$$
on the number of connected components of $V$ was obtained in \cite{DG}.
In case $k \le n$, the number of maximal faces of transversal intersections of tropical hypersurfaces
defined by polynomials $f_i$, $1 \le i \le k$,
was expressed in \cite{BB, ST} in terms of mixed Minkowski volumes of Newton polytopes of polynomials $f_i$.
In \cite{Bihan}, for $k=n$, this number was bounded from above using a new concept of {\em discrete mixed volume}
for sparse tropical polynomials.

The structure of this paper is as follows.

In Section~1 we show that any homotopy type of a polyhedral complex is realizable by a tropical prevariety,
so the problem of bounds on Betti numbers is not futile.

In Section~2 we prove the upper bound
$$2^{(r+1)}r !\ {\rm Vol}_r (P_1+ \cdots +P_k)$$
on the sum of Betti numbers of a tropical prevariety $V$,
where $r$ is the dimension of the Minkowski sum $P_1 + \cdots +P_k$ of
Newton polytopes $P_1, \ldots ,P_k$ of tropical polynomials $f_1, \ldots ,f_k$,
and ${\rm Vol}_r (P_1+ \cdots +P_k)$ is the $r$-dimensional volume of the Minkowski sum.
In terms of $\max_i \deg f_i=d$  this implies the bound
\begin{equation}\label{eq:degree_bound}
2^{(n+1)}(kd)^n.
\end{equation}
Note that in \cite{DG} a naive upper bound
$$
\left(\binom{k+7n}{3n}d^{3n} \right)^n
$$
was mentioned.
Also in \cite{DG}, an example is constructed of a prevariety defined by $kn$ polynomials of degrees at most
$d$, containing $(kd/4)^n$ zero-dimensional connected components.
Comparing this lower bound with (\ref{eq:degree_bound}), we note a gap within a factor $n^n$.
An interesting challenge is to close this gap.

In Section~3 we assume that each tropical polynomial $f_i$, $1 \le i \le k$ is $m$-sparse, i.e., consists of
at most $m$ monomials.
In this setting we prove the upper bound
$$
n2^n \binom{k \binom{m}{2}}{n}
$$
on the sum of Betti numbers of $V$.
We give an example of a prevariety, with $k=n$, for which this bound is close to sharp up to the factor $m^n$.

Results in Section~2 can be related to the classical upper bounds on Betti numbers of real algebraic and semi-algebraic sets
obtained by Petrovskii, Oleinik, Milnor and Thom, and developed further by various authors.
In particular, Milnor \cite{Milnor} proved that if a semi-algebraic set $X \subset \Real^n$ is defined by
a system of $k$ non-strict polynomial inequalities of degrees less than $d$, then the sum of Betti numbers of $X$
is at most $(ckd)^n$ for an absolute constant $c>0$ (compare with (\ref{eq:degree_bound})).
Note that in the case when $X$ is defined by an arbitrary Boolean combination of inequalities, the $i$th Betti number
of $X$ does not exceed $(c \nu k d)^n$ for $\nu= \min \{i+1,n-i,k\}$ and an absolute constant $c>0$ \cite{GV}.

The bound in Section~3 can be viewed as a tropical counterpart of the bounds on Betti numbers for fewnomials
\cite{Khovanskii, GVsurvey, BS}.

In Appendix we prove a technical statement needed to reduce the polyhedral complex associated with a prevariety $V$
to a bounded polyhedral complex.

\section{Tropical prevarieties and polyhedral complexes}

It is well known \cite{MS, RST} that any tropical prevariety $V:=V(f_1, \ldots,f_k) \subset \Real^n$
of all common tropical zeroes of tropical polynomials $f_1, \ldots ,f_k$ is
the union of cells of a finite polyhedral complex in $\Real^n$ (in Lemmas~\ref{le:interior}, \ref{le:subset}
below, we provide an explicit representation of $V$ as a complex).
For this complex we will use the same notation, $V$.
There is an understanding among experts (see, e.g., \cite{BJS}) that any homotopy type of a polyhedral complex
is realizable by a tropical prevariety.
In the present section we make this statement precise.

Let $A \subsetneqq \Real^n$ be an affine subspace, defined by linear equations with integer coefficients,
and $A_{\ge}$ be a closed half-space in $A$ defined by a non-strict linear inequality with integer coefficients.

\begin{lemma}\label{le:half}
The half-space $A_{\ge}$ is a tropical prevariety in $\Real^n$.
\end{lemma}

\begin{proof}
Let $A \subsetneqq \Real^n$ be defined by a system of equations $M_i=0,\ 1 \le i \le s$, and
$A_{\ge}$ by a system $M_i=0,\ 1 \le i \le s,\ L \ge 0$, where all $M_i, L$ are linear polynomials with
integer coefficients.
Then $A_{\ge} = V(f_1, \ldots ,f_s,f_{s+1})$, where $f_i,\ 1 \le i \le s+1$, are tropical Laurent polynomials
such that $f_i,\ 1 \le i \le s$, is represented as $\min \{ M_i, 0 \}$, while $f_{s+1}$ as $\min \{ 0, M_1, L \}$.
By adding an appropriate linear polynomial to monomials in $f_i$, we
turn $f_i$ into a tropical polynomial having the same zeroes as $f_i$.
\end{proof}

\begin{corollary}
Any finite union of rational convex polyhedra (in particular, the union of all cells in any rational polyhedral complex)
of positive codimension in $\Real^n$ coincides with a tropical prevariety in $\Real^n$.
\end{corollary}

\begin{proof}
By Lemma~\ref{le:half}, any rational convex polyhedron of positive codimension in $\Real^n$ is a tropical prevariety.
It remains to note that the union of a finite number of tropical prevarieties is also a tropical prevariety
(this is proved by a straightforward analogy with the classical case).
\end{proof}

As another consequence of Lemma~\ref{le:half} we mention the following example of a countable
family of tropical prevarieties such that their intersection is not a tropical prevariety.
A standard 2-dimensional disk in a 2-dimensional subspace of $\Real^3$ is not a tropical prevariety, but
can be approximated from outside by a shrinking countable family of rational convex polyhedra.

\section{Betti numbers for dense tropical polynomials}

Let a tropical polynomial $f$ be represented as
$$\min_{1 \le j \le m} \left\{ \sum_{1 \le i \le n} a_{ji}x_i +b_j \right\},$$
where $0 \le a_{ji} \in {\mathbb Z}$ and $b_j \in \Real$.
The {\em Newton polytope} $P \subset \Real^n$ of $f$ is the convex hull of $m$ points $(a_{j1}, \ldots ,a_{jn})$,
$1 \le j \le m$ (see, e.g., \cite{MS}).

Let $E \subset \Real^{n+1}$ be the convex hull of $m$ points $(a_{j1}, \ldots ,a_{jn},b_j)$, $1 \le j \le m$.
The {\em extended Newton polytope} of $f$ is the set
$$Q:= E+ \{(0, \ldots,0,c)|\> 0 \le c \in \Real \} \subset \Real^{n+1},$$
where $+$ denotes the Minkowski addition (see, e.g., \cite[Definition 21]{GP2}).

Consider a tropical prevariety $V:=V(f_1, \ldots,f_k) \subset \Real^n$.
Let $P_i \subset \Real^n$, $1 \le i \le k$, be the Newton polytope of $f_i$, $r$ be the
dimension of the Minkowski sum $P_1+ \cdots +P_k$, and ${\rm Vol}_r (P_1+ \cdots +P_k)$ be its $r$-dimensional volume.
Without loss of generality, we assume that this volume is positive.

\begin{theorem}\label{th:main1}
The number of faces of all dimensions of $V$ does not exceed
$$
(2^{r+1}-1)r !\ {\rm Vol}_r (P_1+ \cdots +P_k).
$$
\end{theorem}

Before proving this theorem, let us extract some corollaries regarding bounds on Betti numbers.

According to Theorem~\ref{th:bounded} in Appendix, the prevariety $V$, being a finite polyhedral complex,
is homotopy equivalent to a bounded polyhedral complex $W$, having no more cells than $V$.
Note that $W$ has the structure of a finite CW complex with polyhedral cells.
Let $\varphi (W)$ be the number of all faces of $W$.
Then, by Theorem~\ref{th:main1},
\begin{equation}\label{eq:faces}
\varphi (W) \le (2^{r+1}-1)r !\ {\rm Vol}_r (P_1+ \cdots +P_k).
\end{equation}

We will use notations ${\rm b}_\nu(X):= {\rm rank}\ H_\nu(X, \Real)$,
where $H_\nu (X, \Real)$ is a singular $\nu$th homology group, and
$${\rm b}(X):= \sum_{0 \le \nu \le \dim X} {\rm b}_\nu(X).$$

Recalling that ${\rm b}(V)={\rm b}(W)$, and applying to the CW complex $W$ the Weak Morse Inequality
${\rm b}(W) \le \varphi (W)$ \cite[Corollary~3.7]{Forman},
we get from (\ref{eq:faces}) the following upper bound.

\begin{corollary}\label{cor:volume}
The sum of Betti numbers of $V$ satisfies the inequality
\begin{equation}\label{eq:volume}
{\rm b}(V) \le (2^{r+1}-1)r !\ {\rm Vol}_r (P_1+ \cdots +P_k).
\end{equation}
\end{corollary}

Let $d= \max_{1 \le i \le k} \deg f_i$.
Then each $P_i$ is contained in the simplex
$$\{ (x_1, \ldots ,x_n) \in \Real^n|\> x_j \ge 0,\ 1 \le j \le n,\ x_1+ \cdots +x_n \le d \}.$$

Hence, in terms of $d$, the inequality (\ref{eq:volume}) can be presented in the following form.

\begin{corollary}
The sum of Betti numbers of $V$ satisfies the inequality
$$
{\rm b}(V) \le (2^{n+1}-1)(kd)^n.
$$
\end{corollary}

\begin{proof}[Proof of Theorem~\ref{th:main1}]
Let $Q_i \subset \Real^{n+1}$, $1 \le i \le k$ be the {\em extended} Newton polytope of $f_i$.
Let $Q$ be the {\em bottom} of $Q_1+ \cdots +Q_k$, which is the set of all points $(\x, a) \in Q_1+ \cdots +Q_k$
such that there are no points $(\x, b) \in Q_1+ \cdots +Q_k$ with $b<a$.
Note that for the projection map $\pi:\> \Real^{n+1} \to \Real^n$ along the last coordinate we have
$\pi (Q)= \pi (Q_1+ \cdots +Q_k)=P_1+ \cdots +P_k$.
The restriction $\pi|_Q$ is injective, $\dim (Q)= \dim (P_1+ \cdots +P_k)=r$.

Let $F$ be a face of $Q$.
Its {\em dual}, $G(F)$, is the set of all supporting hyperplanes for $Q$ (subset of the set of all supporting
hyperplanes for $Q_1+ \cdots +Q_k$) such that their intersections with $Q$ coincide with $F$, and
no hyperplane contains a straight line parallel to $\{ (0, \ldots ,0,c \}|\> c \in \Real \}$.
Then $G(F)$ can be identified with a face of the dual polytope to $Q_1+ \cdots +Q_k$,
and we have $\dim F + \dim G(F)=n$ (see, e.g., \cite{BB, Bihan, ST}).
Observe that $F$ is representable as a Minkowski sum $F=F_1+ \cdots +F_k$, where each $F_i$ is a face of
the bottom of $Q_i$, and $H \cap Q_i=F_i$ for any $H \in G(F)$.
We say that a face $F$ of $Q$ is {\em tropical} if $\dim F_i \ge 1$ for all $1 \le i \le k$.
Then $V$ coincides with the union of polytopes $G(F)$ for all tropical faces $F$ (cf. \cite{BB, ST}).

Decompose each $r$-dimensional face of $Q$ into $r$-dimensional closed simplices without adding new vertices,
and obtain a triangulation of $Q$ (cf. \cite[Theorem~8.2]{P}).
The number of all simplices in this triangulation is not less than the total number of faces of $V$.
For each $r$-dimensional simplex $S$ in the triangulation the $r$-dimensional simplex $\pi(S)$ has integer vertices.
Since the volume of any $r$-dimensional simplex with integer vertices in $\Real^n$ is at least $1/r !$,
we conclude that ${\rm Vol}_r (\pi (S)) \ge 1/r !$.
Therefore, the number of all $r$-dimensional simplices in the triangulation does not exceed
$r !\ {\rm Vol}_r (P_1+ \cdots +P_k)$.
To complete the proof, it remains to notice that the number of all faces of an $r$-dimensional simplex
is $2^{r+1}-1$.
\end{proof}

\section{Betti numbers for sparse tropical polynomials}

In this section we assume that each tropical polynomial $f_i$, $1 \le i \le k$ is $m$-sparse, i.e.,
contains at most $m$ monomials.
In other words, each $f_i$ can be represented as $\min \{L_{i,1}, \ldots ,L_{i,m} \}$,
where $L_{i,1}, \ldots ,L_{i,m}$ are linear functions on $\Real^n$.
Although, by definition, coefficients in $L_{i,j}$ at variables are non-negative integers,
all results in this section hold for arbitrary real coefficients.

Following \cite{GP}, for any subset $B$ of $D:= \{ (i,j)|\> 1 \le i \le k,\ 1 \le j \le m \}$,
consider the polyhedron $U_B$ consisting of all points $\x \in \Real^n$ such that
$$
\min_{1 \le j \le m} \{L_{i,j}(\x) \}=L_{i,j_0}(\x)\ \text{for every}\ (i,j_0) \in B
$$
and
$$
\min_{1 \le j \le m} \{L_{i,j}(\x) \}<L_{i,j_1}(\x)\ \text{for every}\ (i,j_1) \not\in B.
$$
Note that each set $U_B$ is open in its linear hull, and $U_B \subset V$ if and only if for each $1 \le i \le k$
there exist $1 \le j_0 < j_1 \le m$ such that $(i,j_0), (i,j_1) \in B$.

Denote by ${\mathcal L}(X)$ the linear hull of a set $X \subset \Real^n$.

\begin{lemma}\label{le:interior}
The interior of every face $F$ of $U_B$ in ${\mathcal L}(F)$ coincides with $U_{B_1}$ for a suitable
subset $B_1 \subset D$ such that $B \subsetneqq B_1$.
\end{lemma}

\begin{proof}
There exists a subset $B_1 \subset D$ such that $\dim (U_{B_1} \cap F)= \dim (F)$.
We prove that $B_1$ satisfies the requirements of the lemma.
Observe that $U_{B_1}$ is contained in ${\mathcal L}(F)$.
The condition $\min_j \{L_{i,j}(\x) \}=L_{i, j_0}(\x)$ for every $(i,j_0) \in B$
and each $\x \in {\mathcal L}(U_B)$ implies that $(i,j_0) \in B_1$.
Hence, $B \subset B_1$, and obviously $B \neq B_1$.

By \cite[Theorem~4.4]{DG}, for any two points in $U_{B_1}$, their sufficiently small
neighbourhoods are homeomorphic (in fact, isomorphic) by a linear translation from one point to another.
Therefore, $U_{B_1}$ is contained in the interior of $F$ in ${\mathcal L}(F)$.
It remains to show that, conversely, $U_{B_1}$ contains the interior of $F$.

For contradiction, assume that there exists a point $\x$ in the interior of $F$
such that $\x \in \overline{U_{B_1}} \cap U_{B_2}$ for some subset $B_2 \subset D$ different from $B_1$,
where $\overline{U_{B_1}}$ denotes the closure of $U_{B_1}$.
Choose any $(i,j_2) \in B_2 \setminus B_1$, and $(i,j_0) \in B$ such that
$\min_j \{L_{i,j}(\y)\}=L_{i,j_0}(\y)$ for every $\y \in {\mathcal L}(U_B)$.
It follows that $\min_j \{L_{i,j}(\y)\}= L_{i,j_2}(\y)$ for every $\y \in {\mathcal L}(F)$.
We get a contradiction with the assumption that $(i,j_2) \not\in B_1$, hence $U_{B_1}$ contains
the interior of $F$ in ${\mathcal L}(F)$.
\end{proof}

\begin{remark}\label{re:remark}
Being a polyhedron, the set $\overline{U_B}$ coincides with
$$
\{ \x \in \Real^n|\> \min_j \{L_{i,j}(\x) \}=L_{i,j_0}(\x),\ 1 \le i \le k\ \text{for every}\
(i,j_0) \in B \}.
$$
\end{remark}

\begin{lemma}\label{le:subset}
For any subsets $B_1, B_2 \subset D$ there exists a subset $B \subset D$ such that $B \supset (B_1 \cup B_2)$
and $\overline{U_B} = \overline{U_{B_1}} \cap \overline{U_{B_2}}$.
\end{lemma}

\begin{proof}
For any $(i,j_0) \in B_1 \cup B_2,\> 1 \le i \le k$ and any $\x \in \overline{U_{B_1}} \cap \overline{U_{B_2}}$
we have $\min_j \{ L_{i,j}(\x) \}=L_{i,j_0}(\x)$.

Define $B$ as the set of all $(i,j_1) \in D$ such that
$\min_j \{ L_{i,j}(\x) \}=L_{i,j_1}(\x)$ for every $\x \in \overline{U_{B_1}} \cap \overline{U_{B_2}}$.
Hence, $B \supset (B_1 \cup B_2)$.
It remains to prove that $\overline{U_B} = \overline{U_{B_1}} \cap \overline{U_{B_2}}$.

The inclusion $\overline{U_B} \subset (\overline{U_{B_1}} \cap \overline{U_{B_2}})$ follows from
Remark~\ref{re:remark}.
Conversely, since $\overline{U_{B_1}} \cap \overline{U_{B_2}}$ is a closed convex polyhedron,
for every $(i,j_2) \in D \setminus B$ the set of all $\x \in \overline{U_{B_1}} \cap \overline{U_{B_2}}$
such that $\min _j \{ L_{i,j}(\x)\}< L_{i,j_2}(\x)$ contains the interior of
$\overline{U_{B_1}} \cap \overline{U_{B_2}}$.
Hence, $U_B$ also contains this interior.
It follows that $(\overline{U_{B_1}} \cap \overline{U_{B_2}}) \subset \overline{U_B}$.
\end{proof}

Similar to \cite{DG}, consider an arrangement $\mathcal A$ in $\Real^n$ consisting of at most
$\ell:= k \binom{m}{2}$ hyperplanes of the form $L_{i,j_1}=L_{i,j_2}$ for all $1 \le i \le k,\>  1 \le j_1<j_2 \le m$.
Without loss of generality, assume that $n \le \ell$.

Observe that for every $B \subset D$ the set $U_B$ is a face of $\mathcal A$ whenever $U_B \neq \emptyset$.
Then Lemmas~\ref{le:interior} and \ref{le:subset} imply that the number of faces $\varphi (V)$ of $V$
does not exceed the number of faces $\varphi (\mathcal A)$ of $\mathcal A$.
According to \cite{Z}, $\varphi (\mathcal A) \le n2^n \binom{\ell}{n}$, thus
$\varphi (V) \le n2^n \binom{\ell}{n}$.
We proved the following theorem.

\begin{theorem}\label{th:main2}
The number of all faces of a tropical prevariety $V \subset \Real^n$ defined by $k$ $m$-sparse tropical polynomials
is at most
$$
n2^n \binom{k \binom{m}{2}}{n}.
$$
\end{theorem}

As in the proof of Corollary~\ref{cor:volume}, passing from $V$ to a homotopy equivalent bounded polyhedral complex $W$,
having no more cells than $V$ has, we obtain the following corollary.

\begin{corollary}\label{cor:sparse}
The sum of Betti numbers of $V$ satisfies the inequality
$$
{\rm b}(V) \le n2^n \binom{k \binom{m}{2}}{n}.
$$
\end{corollary}

In conclusion, we construct an example of a tropical prevariety with $k=n$ which shows that upper bounds in
Theorem~\ref{th:main2} and Corollary~\ref{cor:sparse} differ from a lower bound up to a factor $m^n$.

Let $f_i$, $1 \le i \le n$ be a tropical polynomial in one variable $X_i$, of degree $m$ with $m$
tropical zeroes.
Then the tropical prevariety defined by the system $f_1, \ldots ,f_n$ consists of exactly $m^n$
isolated points.

\section{Appendix:
homotopy equivalence of a finite polyhedral complex to a bounded polyhedral complex}

\begin{theorem}\label{th:bounded}
Let $K$ be a finite polyhedral complex in $\Real^n$ in which some cells may be unbounded.
Then $K$ is homotopy equivalent to a bounded polyhedral complex having no more cells than $K$ has.
\end{theorem}

\begin{lemma}\label{le:poly}
If a convex (closed) polyhedron $P$ in $\Real^n$ contains a straight line $\ell$, and $T$ is
an affine subspace $T$ in $\Real^n$, complementary orthogonal to $\ell$, then $P \simeq \ell \times (P \cap T)$,
where $\simeq$ denotes homeomorphism.
\end{lemma}

\begin{proof}
Observe that for any $\x \in P$, in particular, for any $\x \in P \cap T$, the straight line
$m$ through $\x$, parallel to $\ell$, is also contained in $P$.
\end{proof}

\begin{lemma}\label{le:one}
Let $K$ be a finite polyhedral complex in $\Real^n$, with a cell containing a straight line $\ell$,
and $T$ be an affine subspace $T$ in $\Real^n$, complementary orthogonal to $\ell$.
Then the connected component $C$ of $K$, containing $\ell$ is homeomorphic to $\ell \times (C \cap T)$, where
$C \cap T$ is a $(\dim (K) -1)$-dimensional connected polyhedral complex consisting of all cells of the kind
$P \cap T$ for all cells $P$ in $C$.
\end{lemma}

\begin{proof}
Let $P$ and $Q$ be the cells in $K$, having a common face $R$, and let $P$ contain a straight line $\ell$.
By Lemma~\ref{le:poly}, $P \simeq \ell \times (P \cap T)$.
Since $R \subset Q$, cells $Q$ and $R$ contain straight lines $\ell_Q$ and $\ell_R$, respectively,
parallel to $\ell$.
Hence $Q \simeq \ell \times (Q \cap T)$ and $R \simeq \ell \times (R \cap T)$.
It follows that $P$ and $Q$ are in the same connected component $C$ of $K$ and $P \cup Q \simeq \ell \times (P' \cup Q')$.
\end{proof}

\begin{lemma}\label{le:many}
Let $K$ be a finite polyhedral complex in $\Real^n$, with a cell $D$ containing an affine space $L$ with $\dim (L)=d \ge 0$
and not containing any affine space of dimension greater than $d$.
Let $T$ be an affine space complementary orthogonal to $L$.
Then the connected component $C$ of $K$, containing $D$, is homeomorphic to $L \times (C \cap T)$, where
$C \cap T$ is a $(\dim (K) -d)$-dimensional connected polyhedral complex, not containing any straight line, and
consisting of all cells of the kind $P \cap T$ for all cells $P$ in $C$.
\end{lemma}

\begin{proof}
Proof by induction on $d$, with the base for $d=0$ being obvious.
If $d>0$ then $D$ contains a straight line, which we denote by $\ell$.
Let $T_\ell$ be an affine space complementary orthogonal to $\ell$.
By Lemma~\ref{le:one}, $C \simeq \ell \times (C \cap T_\ell)$, while $\dim (C \cap T_\ell)= \dim (C)-1$.
The cell $D':= D \cap T_\ell$ contains the affine space $L \cap T_\ell$, $\dim (L \cap T_\ell)=d-1$,
and does not contain any affine space of dimension greater than $d-1$.
To conclude the proof, apply the inductive hypothesis to $C \cap T_\ell$.
\end{proof}

\begin{lemma}\label{le:bounded}
Any finite polyhedral complex $K$ in $\Real^n$, is homotopy equivalent to a complex $K'$ in $\Real^n$
in which no cell contains a straight line.
Complexes $K$ and $K'$ have equal number of cells.
\end{lemma}

\begin{proof}
Consider in turn each connected component $C$ of $K$.
Let $C$ contain an affine space $L_C$ with $\dim (L_C) \ge 0$ and not contain
any affine space of dimension greater than $\dim (L_C)$.
By Lemma~\ref{le:many}, $C \simeq L_C \times (C \cap T_C)$, where $T_C$ is an affine space complementary
orthogonal to $L_C$.
Then $C':=C \cap T_C$ is a deformation retract of $C$, not containing a straight line.
Moreover, each cell in $C'$ is of the form $D \cap T_C$, where $D$ is a cell in $C$, hence complexes
$C$ and $C'$ have equal number of cells.
\end{proof}

\begin{proof}[Proof of Theorem~\ref{th:bounded}]
In view of Lemma~\ref{le:bounded}, it is sufficient to prove the theorem in the case when no cell in $K$
contains a straight line.
In particular, every connected component of $K$ is {\em pointed} (i.e., contains a vertex of $K$).

Semi-algebraic local triviality \cite{Hardt} implies that for any sufficiently large positive $\rho \in \Real$
the intersection $\widehat K:= K \cap S_\rho$, and the intersection of $K$ with the interior of $S_\rho$,
for the simplex
$$
S_\rho= \{ (x_1, \ldots ,x_n) \in \Real^n|\> x_j \ge -\rho,\ 1 \le j \le n,\ x_1+ \cdots +x_n \le \rho \},
$$
are both homotopy equivalent to $K$.
Impose on $\widehat K$ a structure of a (bounded) polyhedral complex by adding to all bounded cells in $K$
(they are contained in the interior of $S_\rho$) the new polyhedral cells that are intersections
of unbounded cells in $K$ with faces of $S_\rho$.

Consider an unbounded cell $P$ in $K$.
There is a (closed) facet $R$ of $S_\rho$ such that $P \cap R = \emptyset$.
Suppose that there are exactly $k$ such facets for $P$, remove them from $\partial S_\rho$ and denote the
difference by $(\partial S_\rho)_P$.
Note that $(\partial S_\rho)_P$ has a unique face $L_P$ of the minimal dimension $k-1$ (the intersection of
all facets of $(\partial S_\rho)_P$).
Choose a point ${\bf a}_P \in L_P$ for each cell $P$.

Following \cite[Theorem~8.2]{P}, perform a triangulation of $\widehat K$ by induction on $\dim (\widehat K)$.
In the base case of 1-skeleton of $\widehat K$, its cells are already simplices.
In the case of $m$-skeleton, with $m>1$, and an unbounded cell $P$, consider simplices having ${\bf a}_P$ as one
vertex and other vertices being all vertices of one of the simplices in $(m-1)$-skeleton,
constructed by the inductive hypothesis.
In the case of a bounded $m$-cell $P$, perform the same construction, choosing an arbitrary vertex in $P$,
instead of ${\bf a}_P$.
According to \cite[Theorem~8.2]{P}, we obtain a triangulation of $\widehat K$
which we will continue to denote by $\widehat K$.

We now prove by induction that all simplices in $\widehat K$ with vertices in $\partial S_\rho$ are
contained in $\partial S_\rho$.
Every $m$-simplex in an $m$-cell $P$ in $\widehat K$ with vertices in $\partial S_\rho$ is the union of closed
intervals of the kind $[{\bf a}_P, {\bf b}]$, where each ${\bf b}$ belongs to the same facet of $(\partial S_\rho)_P$,
by the inductive hypothesis.
Since $L_P$ is the intersection of all facets of $(\partial S_\rho)_P$, points ${\bf a}_P$ and ${\bf b}$ belong
to the same facet, hence $[{\bf a}_P, {\bf b}]$ lies in this facet, and so does the whole simplex.

Partition the set of all vertices $V$ of $\widehat K$ into two classes, $V=A \cup B$, where
$A$ are vertices in $K$ while $B$ are vertices belonging to $\partial S_\rho$.
Let $\widehat K_A$ be the induced subcomplex on vertices in $A$, i.e., the complex of bounded cells in $K$,
and $\widehat K_B$ be the induced subcomplex on vertices in $B$.
Note that $\widehat K_A$ is a triangulation of a bounded polyhedral complex having no more cells than $K$ has,
while $\widehat K_B$ is a triangulation of $\widehat K \cap \partial S_\rho$.
Then, by \cite[Lemma~4.7.27]{BV}, $\widehat K_A$ is a deformation retract of
$\widehat K \setminus \widehat K_B$.
It follows that $\widehat K_A$ is homotopy equivalent to $K$, and we can take $\widehat K_A$ as the bounded
polyhedral complex required in the theorem.
\end{proof}

\subsection*{Acknowledgements}
Authors thank D. Feichtner-Kozlov, G.M. Ziegler, and the anonymous referee for useful comments.
Part of this research was carried out during authors' joint visit in September 2017 to the Hausdorff
Research Institute for Mathematics at Bonn University, under the program Applied and Computational Algebraic
Topology, to which they are very grateful.
D. Grigoriev was partly supported by the RSF grant 16-11-10075.

\end{document}